\documentclass[12pt]{amsart}
\usepackage{amsmath,amssymb}
\usepackage[mathscr]{eucal}
\begin{document}

\newcommand{\bbR}{{\mathbb R}}
\newcommand{\bbN}{{\mathbb N}}
\newcommand{\bbI}{{\mathbb I}}
\newcommand{\bbQ}{{\mathbb Q}}
\newcommand{\bbP}{{\mathbb P}}
\def\A{\mathscr{A}}
\def\B{\mathscr{B}}
\def\C{\mathscr{C}}
\def\D{\mathscr{D}}
\def\E{\mathscr{E}}
\def\F{\mathscr{F}}
\def\H{\mathscr{H}}
\def\K{\mathscr{K}}
\def\M{\mathscr{M}}
\def\N{\mathscr{N}}
\def\G{\mathscr{G}}
\def\P{\mathscr{P}}
\def\U{\mathscr{U}}
\def\V{\mathscr{V}}
\def\W{\mathscr{W}}
\renewcommand{\ss}{\subseteq}
\newcommand{\dsize}{\displaystyle}
\def\dom{\mathop{\mathrm{dom}}}
\newtheorem{thm}{Theorem}[section]
\newtheorem{cor}[thm]{Corollary}
\newtheorem{lem}[thm]{Lemma}
\newtheorem{prop}[thm]{Proposition}
\newtheorem{remark}[thm]{Remark}
\newtheorem{ex}[thm]{Example}
\font\sans=cmss10
\def\axiom#1{{\sans (#1)}}
\def\wk{\text{weak}}
\def\domc#1{{\dom c_{#1}}}
\def\sm{\smallsetminus}

\title{On isometric embeddings and continuous maps onto the irrationals}
\author{El\.{z}bieta Pol and Roman Pol}
\address{University of Warsaw \and  University of Warsaw}
\email{E.Pol@mimuw.edu.pl \and  R.Pol@mimuw.edu.pl}
\date{\today}

\begin{abstract}
Let $f : E \to F$ be a continuous map of a complete separable metric space $E$ onto the irrationals. We shall show that if a complete separable metric space $M$ contains isometric copies of every closed relatively discrete set in $E$, then $M$ contains also an isometric copy of some fiber $f^{-1}(y)$. We shall show also that if all fibers of $f$ have positive dimension, then the collection of closed zero-dimensional sets in $E$ is non-analytic in the Wijsman hyperspace of $E$. These results, based on a classical Hurewicz's theorem, refine some results from \cite{PP} and answer a question in \cite{BCZ}.
\end{abstract}

\maketitle

\section{Introduction}

In  \cite{PP} we proved that each complete separable metric space containing isometric copies of every countable complete metric space contains isometric copies of every separable metric space.

 We shall refine this result to the following effect.


\begin{thm} Let $f : E \to F$ be a continuous map of a complete separable metric space onto a non-$\sigma$-compact metric space. Then there exists a relatively discrete set $S$ in $E$ such that, for any complete separable metric space $M$ containing  isometric copies of every subset of $S$ closed  in $E$, some fiber $f^{-1} (y)$ embeds isometrically in $M$.
\end{thm}

The result from \cite{PP} follows from this theorem, if we consider the restriction map $f : C[0,1] \to C[\frac{1}{2}\, , 1 ]$ (recall that by the Banach-Mazur theorem, cf. \cite{FHHMPZ}, Theorem 5.17, the space $(C(I), d_{\sup})$ of all real-valued continuous functions on the interval $I$, equipped with the metric $d_{\sup}(f,g)= \sup \{ \mid \! f(t) - g(t)\! \mid : t \in I \}$, is isometrically universal for all separable metric spaces).

Also, as in \cite{PP}, one can replace in this theorem isometries by uniform homeomorphisms.

The proofs will go along the same lines as in \cite{PP}, and  an essential part of the reasonings can be taken directly from \cite{PP}, cf. sec.4.

 However,  a classical Hurewicz's theorem on non-analyticity of the set of  compact subsets of the rationals is applied in a different way than  in \cite{PP}.
 We shall prove a result based on the Hurewicz theorem in a slightly more general form than needed for Theorem 1.1 in Section 3, to establish a link with some questions concerning the dimension, discussed in Section 5.

\section{Preliminaries}

\subsection{The Effros Borel spaces.}
Our terminology related to the descriptive set theory follows \cite{D} and \cite{K}.
An analytic space is a metrizable continuous image of the irrationals.

Given an analytic space $E$, we denote by $F(E)$ the space of closed subsets of $E$ and  $\B _{F(E)}$ - the Effros Borel structure in $F(E)$,
 is the $\sigma$-algebra in $F(E)$ generated by the sets $\{ A \in F(E): A \cap U \neq \emptyset \}$, where $U$ is open in $E$.

 We shall say that $\A \subset F(E)$ is a Souslin set in the Effros Borel space $(F(E), \B _{F(E)})$ if $\A$  is a result of the Souslin operation on sets from $\B _{F(E)}$.

 If $X$ is a compact metrizable space, we shall consider the hyperspace $F(X)$ with the Vietoris topology and then $\B _{F(X)}$ coincides with the $\sigma$-algebra of Borel sets in the compact metrizable space $F(X)$.

  If $X$ is a compact metrizable extension of an analytic set $E \subset X$, the map $A \to \overline{A}$ (the closure is taken in $X$) from $F(E)$ to $F(X)$ is a Borel isomorphism, with respect to the Effros Borel structures, onto the analytic subspace $\{ \overline{A} : \, A \in F(E) \}$ of the hyperspace $F(X)$ and hence Souslin sets in $F(E)$ are mapped onto analytic sets in $F(X)$, cf. \cite{D}, Section 2. In particular, if $E \subset G \subset X$ and $G$ is analytic, the collection of closures of elements of $F(E)$ in $G$ is a Souslin set in $F(G)$.

 \subsection{The Hurewicz theorem.}

 Let $I = [0,1]$ and let $\bbQ$ be the set of rationals in $I$.

 The classical Hurewicz theorem asserts that any Souslin set in $F(I)$ containing all compact subsets of $\bbQ$, contains an element intersecting $I \setminus \bbQ$.

 We shall derive from this theorem the following observation, which we shall use in the next section.

 Let us arrange points of $\bbQ$ into a sequence $q_{1} , q_{2} , \ldots $ (without repetitions), let
 \smallskip\begin{enumerate}
\item[(2.1)] $\;D = \{ (q_{n}, \frac{1}{m}) : \, n=1,2,\ldots, \; m \geq n \}$, $\;L = (I \setminus \bbQ ) \times \{ 0 \} $,
\end{enumerate}
let
 \smallskip\begin{enumerate}
\item[(2.2)] $\;T = L \cup D $
\end{enumerate}
be the subspace of the plane (notice that $D$ is relatively discrete in $T$), and let
 \smallskip\begin{enumerate}
\item[(2.3)] $\;\D = \{ A \subset D: \; A \hbox{ is closed in } T \}$.
\end{enumerate}

\medskip

\noindent{\bf Lemma 2.2.1.}
{\it For any Souslin set $\E$ in $F(T)$ containing $\D$, some element of $\E$ intersects $L$.}

\begin{proof} For $A \in F(T)$, $\overline{A}$ will denote the closure in the plane. As was recalled in 2.1, the set $\{ \overline{A}: \; A \in \E \}$ is analytic in $F(\overline{T})$ (notice that
$\overline{T}= (I \times \{ 0 \} ) \cup D$), hence the set $\{ (K, \overline{A}) \in F(I) \times F(\overline{T}): \; A \in \E \hbox{ and } K \times \{ 0 \} \subset \overline{A} \}$ is analytic in the product of the hyperspaces, and so is its projection onto $F(I)$,
 \smallskip\begin{enumerate}
\item[(2.4)] $\;\E ^{\star} = \{ K \in F(I): \; K \times \{ 0 \} \subset \overline{A} \hbox{ for some } A \in \E \}$.
\end{enumerate}

If $K \subset \bbQ$ is compact, $A = D \cap (K \times I)$ is closed in $T$ and $ K \times \{ 0 \} \subset \overline{A}$, hence $K \in \E ^{\star}$, cf. (2.4). By the Hurewicz theorem, there is $A \in \E$ such that $\overline{A}$ intersects $L$, cf. (2.1) and (2.4), and
 since  $A$ is closed in $T$, $A$ intersects $L$.
\end{proof}

\subsection{A remark on continuous maps onto the irrationals.}
We shall need the following observation. This is close to some well-known results, but for readers convenience, we shall provide a brief justification.

\medskip

\noindent{\bf Lemma 2.3.1.}
 {\it Let $f : E \to F$ be a continuous map of an analytic space onto a non-$\sigma$-compact metrizable space.
There is a closed copy of the irrationals $P$ in $F$ and continuous maps $g_{n} : P \to E$ such that, for each $t \in P$, $\{ g_{n} (t): \, n = 1,2,\ldots \}$ is a dense subset of $f^{-1} (t)$.}

\medskip

\begin{proof}
Let $p : \bbN  ^{\bbN} \to E$ be a continuous surjection of the irrationals onto the analytic space $E$.

Then $u = f \circ p : \bbN  ^{\bbN} \to F$ is a continuous surjection onto a non-$\sigma$-compact metrizable space and one can find a closed copy of the irrationals $P$ in $F$ such that the restriction map $u \mid u^{-1} (P) : u^{-1} (P) \to P$ is open, cf. \cite{PZ}, proof of Theorem 3.1.

By a selection theorem of Michael \cite{M}, one can define a sequence of continuous selections $w_{n} : P \to u^{-1} (P)$ for the lower-semicontinuous multifunction $t \to u^{-1} (t)$ such that, for each $t \in P$, the set $\{ w_{n} (t): \, n=1,2, \ldots \}$ is dense in $u^{-1}(t)$.

Then the functions $g_{n} = p \circ w_{n} : P \to f^{-1} (P)$ satisfy the assertion.
\end{proof}

\section{An application of the Hurewicz theorem.}
The following proposition strengthens a known fact that, for the irrationals $\bbN ^{\bbN}$, any Souslin set in $F(\bbN ^{\bbN} )$ containing all countable closed sets in $\bbN ^{\bbN}$, contains also a non-$\sigma$-compact set (this is stated in \cite{K}, Exercises 27.8, 27.9; to derive this fact from the proposition, notice that $\bbN ^{\bbN}$ is homeomorphic to $\bbN ^{\bbN} \times \bbN ^{\bbN}$ and consider the projection $\bbN ^{\bbN} \times \bbN ^{\bbN} \to \bbN ^{\bbN}$).

The setting is a bit more general than needed for Theorem 1.1, but it is useful to establish connections with some topics in the dimension theory, discussed in Section 5.

\begin{prop}
Let $f : E \to F$ be a continuous map of an analytic space onto a non-$\sigma$-compact metrizable space. Then  there exists a relatively discrete set $S$ in $E$ such that for any Souslin set $\A$ in $F(E)$  containing all subsets of $S$ closed  in $E$, there are $A \in \A$ and $y \in F$ with $f^{-1}(y) \subset A$.
\end{prop}
\begin{proof} Let $P$ be a closed copy of the irrationals in $F$ and $g_{n} : P \to E$ continuous maps described in Lemma 2.3.1, and let $T = L \cup D$ be the subspace of the plane defined in (2.1) and (2.2).

Since $T$ is a zero-dimensional $G_{\delta}$-subset of the plane, there
is a homeomorphic embedding
\smallskip\begin{enumerate}
\item[(3.1)] $\;h : T \to P$, $\; h(T)$ closed in $P$.
\end{enumerate}
Let us arrange points of $D$ into a sequence without repetitions
\smallskip\begin{enumerate}
\item[(3.2)] $\;D = \{ d_{1}, d_{2}, \ldots \}\;$ and  $\;c_{n} = h(d_{n})$.
\end{enumerate}

We shall check that, cf. (3.2),
\smallskip\begin{enumerate}
\item[(3.3)] $\;S = \{ g_{m} (c_{n}): \; n=1,2, \ldots , \; m \leq n \} \subset E$
\end{enumerate}
satisfies the assertion of the proposition.

Since $g_{m} (c_{n}) \in f^{-1} (c_{n})$, $\, f(S) = h(D)$ is relatively discrete and $S$ intersects each fiber of $f$ in at most finite set, cf. (3.3). Therefore $S$ is relatively discrete.

Let, for $X \in F(T)$,
\smallskip\begin{enumerate}
\item[(3.4)] $\; \varphi (X) = f^{-1} (h(X \cap L)) \cup (S \cap f^{-1}(h(X \cap D)))$.
\end{enumerate}
Since all accumulation points of $S$ in $E$ are in $f^{-1} (h(L))$ and $h(X)$ is closed in $F$, cf. (3.1), we have $\varphi (X) \in F(E)$.

We shall check that
\smallskip\begin{enumerate}
\item[(3.5)] $\; \varphi : F(T) \to F(E)$ is Borel,
\end{enumerate}
with respect to the Effros Borel structure.

To that end, let us fix an open set $U$ in $E$, and let
\smallskip\begin{enumerate}
\item[(3.6)] $\; \U = \{ X \in F(T): \; \varphi (X) \cap U \neq \emptyset \}$.
\end{enumerate}

Let $X \in \U$. If for some $m \leq n$, $\, d_{n} \in X$ and $\, g_{m}(c_{n}) \in U$, cf. (3.2), (3.3), (3.4), the element $\{ Y \in F(T): \; d_{n} \in Y \}$ of $\B _{F(T)}$ contains $X$ and is contained in $\U$.

Let $a \in X \cap L$ and $f^{-1} (h(a)) \cap U  \neq \emptyset $. Since the points $g_{m} (h(a))$ are dense in $f^{-1}(h(a))$, there is $m$ such that $g_{m} (h(a)) \in U$. Let $V$ be a neighbourhood of $h(a)$ in $F$ such that $g_{m} (V) \subset U$, and let us pick a rectangle $\, J = (r,s) \times [0,\frac{1}{p})$ disjoint from $\, \{ d_{1}, \ldots , d_{m} \}$ with $r,s \in \bbQ$, containing $a$, such that $h(J \cap T) \subset V$. If $Y \in F(T)$ hits $J$, there is either $b \in Y \cap L$ with $h(b) \in V$ and then $f^{-1}(h(b)) \subset \varphi (Y)$ intersects $U$, or there is $d_{n} \in Y \cap J$ with $n > m$ and then, cf. (3.3), (3.4), $g_{m} (c_{n}) \in \varphi (Y) \cap U$.

Therefore the element $\{ Y \in F(T): \; Y \cap J \neq \emptyset \}$ of  $\B _{F(T)}$ contains $X$ and is contained in $\U$.

We demonstrated that $\U$ is a countable union of elements of  $\B _{F(T)}$, hence belongs to the Effros Borel structure of $F(T)$.

Having checked (3.5), let us consider the set
\smallskip\begin{enumerate}
\item[(3.7)] $\; \mathscr{S} = \{ A \subset S: \; A \in F(E) \}$
\end{enumerate}
and let
\smallskip\begin{enumerate}
\item[(3.8)] $\; \mathscr{S} \subset \A$, $\; \A$ is Souslin in $(F(E), \B _{F(E)})$.
\end{enumerate}

By (3.5),
\smallskip\begin{enumerate}
\item[(3.9)] $\E = \varphi ^{-1} ( \A)$ is Souslin in $(F(T), \B _{F(T)})$.
\end{enumerate}
If $X \subset D$ is closed in $T$, $h(X)$ is closed in $F$ and $\varphi (X)$ is closed in $E$, cf. (3.4), hence $\varphi (X) \in \mathscr{S}$, cf. (3.7). Therefore, by (3.8), for the set $\D$ defined in  (2.3), we have $\D \subset \E$ and Lemma 2.2.1 provides $X \in \E$ and a point $a \in X \cap L$. By (3.4) and (3.9) we get $A = \varphi (X) \in \A$ and $f^{-1} (h(a)) \subset A$.
\end{proof}

\section{Proof of Theorem 1.1}
We shall recall briefly some reasonings from \cite{PP} to derive this theorem from Proposition 3.1.

Given $f : E \to F$ as in this theorem, let us pick $S$ satisfying the assertion of Proposition 3.1.

Let $e$ be the complete metric on $E$ and let $(M,d)$ be any complete separable metric space, containing isometric copies of every subset of $S$ closed in $E$. Let
$$\H = \{ T \in F(E \times M): \; \hbox{ for every } (x_{1},y_{1}), (x_{2}, y_{2} ) \in T, \; e(x_{1}, x_{2} ) = d(y_{1}, y_{2} ) \}.$$
One checks, cf. \cite{PP}, page 193, that $\H$ is in $ \B _{F(E \times M)}$ and the map $T \to \pi (T)$ associating to $T \in \H$ its projection onto $E$ is a Borel map $\pi : \H \to F(E)$.

Therefore $\A = \pi (\H ) $ is a Souslin set in $(F(E), \B _{F(E)})$. If $X \subset S$ is closed in $E$, there is an isometry $f : X \to f(X) \subset M$ and the graph of $f$ is an element of $\H$.

It follows that the Souslin set $\A$ contains all subsets of $S$ closed in $E$, and by the choice of $S$, some $A \in \A$ contains a fiber $f^{-1}(y)$.

Now, $A = \pi (T)$ and $T$ is the graph of an isometry that embeds $A$ in $M$. In effect,  $f^{-1}(y)$ embeds isometrically in $M$.

\section{The collections of zero-dimensional sets in Effros Borel spaces.}
Our terminology concerning the dimension theory follows \cite{vM}.

Given an analytic space, we shall write
\smallskip\begin{enumerate}
\item[(5.1)] $\; F_{0}(E) = \{ A \in F(E): \; \hbox{ dim}A = 0 \}$.
\end{enumerate}

We shall derive from Proposition 3.1 the following result.
\begin{prop}
Let $E$ be an analytic space that admits a continuous map $f : E \to F$  onto a non-$\sigma$-compact metrizable space such that all fibers $f^{-1}(y)$ have positive dimension.
 Then  for any analytic extension $G$ of $E$ with dim$(G \setminus E) \leq 0$, the set $F_{0}(G)$ is not  Souslin in the Effros Borel space $(F(G), \B _{F(G)})$.
\end{prop}
\begin{proof}
By Proposition 3.1, there is a relatively discrete set $S$ in $E$ such that for any Souslin set $\A$ in $F(E)$ containing $\mathscr{S} = \{ A \in F(E): \, \emptyset \neq A \subset S \}$, some element of the set $\A$ contains a fiber of $f$ and hence has positive dimension.

 Now, consider an analytic extension $G$ of $E$ with  dim$(G \setminus E) \leq 0$ and, aiming at a contradiction assume that
$F_{0}(G)$ is  Souslin in $F(G)$. As was recalled in sec. 2.1, the map $A \to \overline{A}$ from $F(E)$ to $F(G)$ is Borel, and hence we would get that the set $\A = \{ A \in F(E): \; \hbox{ dim}\overline{A} \leq 0 \}$ is Souslin in $F(E)$.

If $A \in \mathscr{S}$, then $A$ is a relatively discrete closed set in $E$, and hence $\overline{A} \setminus A$ is a closed subset of $G$ contained in $G \setminus E$. This implies that dim$\overline{A} = 0$, i.e., $\mathscr{S} \subset \A$. However, all members of $\A$ are zero-dimensional, which contradicts properties of $\mathscr{S}$.
\end{proof}

In particular, if $\bbP$ is the set of irrationals in $I = [0,1]$,
\smallskip\begin{enumerate}
\item[(5.2)] $\; F_{0} ( \bbP \times I)$ is not Souslin in $F(\bbP \times I)$
\end{enumerate}
(this rectifies a remark in \cite{CGK}, \S 3.A).

Banakh, Cauty and Zarichnyi \cite{BCZ}, Question 9.12, asked about the Borel type of the collection $F_{0}(E)$ in the space $CL(E) = F(E) \setminus \{ \emptyset \}$, when $E$ is a completely metrizable separable, and $CL(E)$ is considered with the Wijsman topology $\tau _{W}$, determined by some metric $d$ generating the topology of $E$ (i.e., $\tau _{W}$ is the weakest topology making all functionals $ A \to {\rm dist} (z,A)$, $z \in E$,  continuous), cf. \cite{B}.

The Wijsman hyperspace $(CL(E), \tau _{W})$ is completely metrizable, separable, cf. \cite{C}, and the  Borel sets with respect to $\tau _{W}$ coincide with the members of the Effros Borel structure in $CL(E)$.
Therefore,
\smallskip\begin{enumerate}
\item[(5.3)] $\; F_{0} ( \bbP \times I)$ is not a Borel (or even  Souslin) set  in $CL(\bbP \times I)$.
\end{enumerate}

One can check that its complement $F(\bbP \times I) \setminus F_{0}(\bbP \times I)$ is Souslin.
Let us consider, however, the subspace $I^{2} \setminus \bbQ ^{2}$ of the square, $\bbQ = I \setminus \bbP$. Since $(I^{2} \setminus \bbQ ^{2}) \setminus (\bbP \times I) = \bbQ \times \bbP$ is zero-dimensional, also $F_{0}( I^{2} \setminus \bbQ ^{2})$ is not Souslin in $F( I^{2} \setminus \bbQ ^{2})$, by Proposition 5.1. But it is not clear to us whether
$F( I^{2} \setminus \bbQ ^{2}) \setminus F_{0}( I^{2} \setminus \bbQ ^{2})$ is Souslin.

In fact, we do not know an answer to the following general question.

\medskip

 \noindent {\bf Question 5.2.} {\it Does there exist an analytic space $E$ such that $F(E) \setminus F_{0}(E)$ is not Souslin in the Effros Borel structure?}

\medskip

This question is related to the following problem, asked in \cite{P1}, where countable-dimensional spaces are countable unions of zero-dimensional spaces.

\medskip

 \noindent {\bf Problem 5.3.} {\it Is the collection $\C$ of all countable-dimensional compact sets in the Hilbert cube $I^{\bbN}$ coanalytic in the hyperspace $F(I^{\bbN})$ equipped with the Vietoris topology? }

\medskip

To see the link between these two questions, let us consider a Borel set $E \subset I^{\bbN}$ such that $I^{\bbN} \setminus E$ is countable-dimensional and all countable-dimensional subsets of $E$ are at most zero-dimensional, cf. \cite{P2}. We shall assume in addition that $E$ is disjoint from the set $\Sigma $ consisting of points in $I^{\bbN}$ with all but finitely many coordinates zero.

By \cite{BP}, Ch.V, \S 5 , there is a homeomorphism $h : I^{\bbN} \setminus \Sigma \to \bbR ^{\bbN} \times \bbR^{\bbN}$ ($\bbR$ - the real line), let $p : \bbR ^{\bbN} \times \bbR^{\bbN} \to
\bbR ^{\bbN} $ be the projection and let $f = p \circ h \mid E : E \to \bbR ^{\bbN} $. Then $f$ is a continuous surjection whose all fibers are uncountable-dimensional. Therefore, by Corollary 5.1, $F_{0}(E)$ is not Souslin in the Effros Borel space.

We do not know if the set $\E = F(E) \setminus F_{0}(E)$ is Souslin. Let us show, however, that if this is the case,  $\C$ in Problem 5.3 is coanalytic.

Suppose that $\E$ is Souslin in $(F(E), \B _{F(E)})$. Then, as was noticed in Section 2.1, the collection $\E ^{\star } = \{ \overline{A}: \; A \in \E \}$ of the closures in $I^{\bbN}$
is analytic in the hyperspace $F ( I^{\bbN})$. Now, $K \in F(I^{\bbN}) \setminus \C$ if and only if $K \cap E$ is uncountable-dimensional, which is equivalent to $K \cap E \not \in F_{0}(E)$.
Therefore, $F(I^{\bbN} ) \setminus \C$ is the projection of the analytic set $\{ (K,L) \in F(I^{\bbN} ) \times F(I^{\bbN} ): \; L \subset K \hbox{ and } L \in \E ^{\star } \}$, hence it is analytic.

\end{document}